\documentclass[a4paper,10pt,reqno]{amsart}
\usepackage[utf8]{inputenc}
\usepackage{amssymb}
\usepackage{amsthm}
\usepackage{amsmath}
\usepackage{tikz}
\usepackage{mathrsfs}
\usepackage{epsf}
\usepackage{pdfpages}
\usepackage{bookmark,hyperref}
\usepackage{stmaryrd}
\hypersetup{
     colorlinks=true,
     linkcolor=blue,
     filecolor=blue,
     citecolor = blue,
     urlcolor=cyan,
     }
\usepackage{pstricks}

\newtheorem{theorem}{Theorem}[section]
\newtheorem{lemma}[theorem]{Lemma}
\newtheorem{corollary}[theorem]{Corollary}
\newtheorem{proposition}[theorem]{Proposition}

\newtheorem{remark}[theorem]{Remark}
\newtheorem{definition}{Definition}
\newtheorem*{theorem*}{Theorem}
\newtheorem*{proposition*}{Proposition}

\numberwithin{equation}{section}

\newtheoremstyle{mythm}%
{3pt}
{3pt}
{}
{}
{\bfseries}
{.}
{.5em}
{}%

\theoremstyle{mythm}
\newtheorem*{notation}{Notation}

\def\Z{\mathbb{Z}}

\def\R{\mathbb{R}}
\def\C{\mathbb{C}}

\DeclareMathOperator{\Ima}{Im}
\DeclareMathOperator{\hess}{Hess}
\DeclareMathOperator{\tr}{tr}

\def\set#1{\left\{\, #1 \,\right\}}
\def\abs #1{\left| \,#1\, \right|}
\def\norm #1{\left\| \,#1\, \right\|}
\def\inner #1#2{\langle \,#1,#2\, \rangle}
\def\cinner #1#2{\langle \,#1\,,\,#2\, \rangle_\C}
\def\cj #1{\overline{#1}}
\def\vs{\vspace{.2cm}}
\def\eval#1#2{\left. #1\right|_{#2}}

\def\calA{\mathcal{A}}
\def\calB{\mathcal{B}}
\def\calC{\mathcal{C}}
\def\calD{\mathscr{D}}
\def\calE{\mathcal{E}}

\def\calL{\mathcal{L}}
\def\calM{\mathscr{M}}

\title[Decomposition of the Jacobi equation]
{A natural decomposition of the Jacobi equation\\
for some classes of $N$-body problems}
\author{Renato Iturriaga}
\address{CIMAT -- Centro de Investigación en Matemáticas, México}
\email{renato@cimat.mx}
\author{Ezequiel Maderna}
\address{CIMAT -- Centro de Investigación en Matemáticas, México}
\email{ezequiel.maderna@cimat.mx}

\date{\today}

\begin{document}

\begin{abstract}
We consider several $N$-body problems.
The main result is a very simple and natural
criterion for decoupling the Jacobi equation for some classes of them.
If $E$ is a Euclidean space, and the potential function $U(x)$
for the $N$-body problem is a $C^2$ function 
defined in an open subset of $E^N$,
then the Jacobi equation along a given motion $x(t)$
writes $\ddot J=HU_x(J)$, where the endomorphism $HU_x$ of $E^N$
represents the second derivative of the potential
with respect to the mass inner product.

Our splitting in particular applies 
to the case of homographic motions by central configurations.
It allows then to deduce the well known Meyer-Schmidt decomposition
for the linearization of the Euler-Lagrange flow in the phase space,
formulated twenty years ago to study the relative equilibria
of the planar $N$-body problem.
However, our decomposition principle applies
in many other classes of $N$-body problems,
for instance to the case of isosceles three body problem,
in which Sitnikov proved the existence of oscillatory motions.

As a first concrete application, for the classical three-body problem
we give a simple and short proof of a theorem of Y. Ou,
ensuring that if the masses verify
$\mu=(m_1+m_2+m_3)^2/(m_1m_2+m_2m_3+m_1m_3)<27/8$
then the elliptic Lagrange solutions are linearly unstable
for any value of the eccentricity.
\end{abstract}

\maketitle

\section{Introduction}

In 1843, Gabriel Gascheau \cite{Gascheau} discusses his thesis
on mechanics, showing the linear stability of Lagrange's
relative equilibria of the three-body problem
under the mass hypothesis
\[
\mu=(M+m+m')^2/(Mm+Mm'+mm')>27\,.
\]
If the three masses are equal then the value
of this constant, which today we know as Gascheau's constant,
is $\mu=3$ and the motion is unstable.
Therefore, to achieve stability in this case,
the presence of a dominant mass is required, that is,
one that is much larger than the others.
This is the first known example that supports
a conjecture of Moeckel \cite{MoeckelPCMaderna},
which states that in the $N$-body problem,
with $N\geq 3$, there is no possibility of
stable motions unless there is a dominant mass.
See \cite{Moeckel}, and problems 15 and 16 in the list
compiled by Albouy, Cabral, and Santos \cite{AlbouyCabralSantos}.

The Gascheau's theorem is an absolutely surprising result for his time.
Firstly, because it has a notable precedent, which is
the announcement one year before by Liouville \cite{Liouville}
that collinear homographic motions are always unstable,
which is why they cannot be observed in nature.
On the other hand, it is surprising that his method
allows him to establish exactly the threshold of stability
in terms of a simple function of the masses that was unknown until then.
Indeed, Gascheau shows that
there is linear stability if and only if $\mu>27$.

The required relationship between the masses
holds for the Sun-Earth-Moon system,
since the mass of the Sun is $M\simeq33.10^4\,m$,
where $m$ is the mass of the Earth,
and $m\simeq100\,m'$ where $m'$ is the mass of the Moon.
However, the motion of this system is clearly very far
from equilateral configurations.
Actually, a linearly stable behavior was not observed
until much later, when the first Trojan asteroids were
discovered at the beginning of the 20th century,
see for instance the historical notes by Robutel
and Souchay in \cite[pp. 198--201]{RobutelSouchay}.
Despite this, in his note in the Comptes Rendus
of the same year \cite{Gascheau2},
Gascheau rightly states that while its study
\emph{
may not be of interest to astronomers, it cannot be
completely unworthy of the attention of geometers,
due to its simplicity and symmetries}.

These Lagrange's relative equilibria,
in which the three masses rotate around their center of mass,
keeping their three mutual distances constant and equal,
are the special case of zero eccentricity within the family of
homographic motions with equilateral configuration,
also known as the elliptic Lagrange's solutions.
In these solutions the equilateral triangle rotates
but also expands and contracts, and each of the bodies
moves as if it were subjected to a Keplerian central force,
that is, traveling along an ellipse.
Actually their motions define an ellipse in the space
of configurations, which is contained in a subspace
of equilateral triangles.

The study of the linear stability of elliptic
Lagrange's solutions is much more recent.
The elliptical case with three positive
masses\footnote{\,The restricted problem
(i.e. the case of a zero mass)
was widely studied in the meantime.}
was first treated numerically by Danby in \cite{Danby}.
The problem was then studied by Roberts in \cite{Roberts},
managing to analytically determine a boundary
for the stability region in the
parameter space
\[
(\,\mu\,,\,e\,)\in R,\quad\quad\,R=[3,+\infty)\times[0,1),
\]
where $e$ denotes the
common eccentricity of the trajectories.
It should be noted that usually authors use
the equivalent parameter $\beta=27/\mu$
in order to have stability for $\beta\in(0,1)$
in the circular case $e=0$.
However, for eccentricity values close to $1$,
Roberts' work makes use of numerical estimates.

An important breakthrough in the study of this
problem was the discovery by Meyer and Schmidt
\cite{MeyerSchmidt} of the phase-space decomposition
of the linearization of the flow,
which is based in what they called the symplectic
coordinates associated to a given central configuration.
Using these coordinates, they recover Roberts' results
and analyze the relative equilibria associated with
other central configurations.

As we will show in section \ref{sec:applications.list},
our decomposition principle in the case of planar
homographic motions is completely equivalent to the
Meyer-Schmidt construction.
However, our decomposition emanates from a simple geometric
consideration in configuration space:
\emph{the Meyer-Schmidt decomposition turns out to be
a consequence of the uniform decomposition
for the Hessian of the Newtonian potential,
which also occurs in contexts other
than homographic motions}.

Before giving a description of the main results
of this paper, we must mention that thanks to
the Meyer-Schmidt decomposition it was possible
to put into practice in this kind of problems the
Morse index theory, as well as the
$\omega$ index theory for symplectic paths,
giving rise to a description of the stability zone
(according to the different spectral types)
in the whole rectangle in which the parameter
$(\mu,e)$ varies. These important advances
are due (for instance) to the works of Hu and Sun \cite{HuSun},
Hu, Long and Sun \cite{HuLongSun}, Hu and Ou \cite{HuOu, HuOu2},
Barutello, Jadanza and Portaluri \cite{BarutelloJadanzaPortaluri}
for the circular case with other potentials,
and see also Hu, Long and Ou \cite{HuLongOu}
for the case of a regular $n$-gon with an additional mass in the center.
However, these techniques have not allowed to obtain
analytical expressions for the curves that border
the different regions into which the parameter space
is decomposed, as they were numerically described
by Martínez, Samà and Simó in \cite{MartinezSamaSimo}
see \cite[Figure 1]{HuLongSun}.
These numeric estimations are based on analytical
information that they deduce for the cases $e>0$ small,
and $e$ close to $1$.

A result that does not depend on numerical estimates,
which is based on Maslov-type index theory,
is the following. In particular it implies
linear instability for any value of the eccentricity.

\begin{theorem*}[Ou \cite{Ou}, 2014]
If $\mu<27/8$ then the essential part of
the Meyer-Schmidt decomposition of any
elliptic Lagrangian solution is hyperbolic.
\end{theorem*}

Now we describe our main results.
First of all, it is worth making the following comment,
which is key to our point of view.
If we consider an $N$-body problem, with bodies of different masses,
which evolve in some Euclidean space $E$,
then the space $E^N$ of configurations
-- which is a mathematical tool for describing a physical phenomena --
has a natural Euclidean structure which is not the product one,
but the one defined by the mass inner product.

We will say that a linear subspace $V\subset E^N$ is invariant
for a given $N$-body problem in $E$
when, for any $x\in V$ where the potential is defined,
the configuration of accelerations $\nabla U(x)$ is also in $V$.

\begin{lemma}
\label{lemma:Jacobi.splitting}
Given an $N$-body problem in a Euclidean space $E$,
and an invariant subspace $V\subset E^N$, 
the Jacobi equation $\ddot J= HU_x(J)$ along any motion
$x(t)\in V$ splits for the orthogonal sum
$V\oplus V^\perp$.
\end{lemma}

Note that $V^\perp$ does not have to be an invariant
subspace for the problem; in fact, in general it is not,
as we will see in the examples.
The lemma states that both $V$ and $V^\perp$
are invariant for $HU_x$, and this for each $x\in V$.
It turns out that for any plane central configuration $x_0$,
the space $S_{x_0}$ of positively similar configurations
is an invariant subspace. 

\begin{definition}
\label{def:cc.strong.ND}
A planar central configuration $x_0$
is strongly non-degenerate
if the restriction of $HU_{x_0}$ to $S_{x_0}^\perp$
is positive definite.
\end{definition}

We will prove the following sufficient condition for
linear instability.

\begin{theorem}
\label{thm:StrongND->unstable}
If $x_0$ is a strongly non-degenerate central configuration,
then any homographic elliptic motion by $x_0$ is
linear unstable.
\end{theorem}

In order to apply the theorem to the planar three-body problem,
we compute the orthogonal to the equilateral configurations,
and we obtain the following result,
which combined with the previous one gives a new proof
of Ou's theorem above.

\begin{theorem}
\label{thm:LagrangeisSND.iff.mu<27/8}
For the classical three-body problem, the equilateral
configuration is strongly non-degenerate if and only
if the Gascheau constant satisfies $\mu<27/8$.
\end{theorem}

Furthermore, we will also show in
Section \ref{sec:definitions.StrongND.equivalent}
that the above definition of strongly non-degenerate
central configuration is equivalent to the definition
given by Hu and Ou in \cite[p. 807]{HuOu}.
Since in the proof of Theorem \ref{thm:StrongND->unstable}
we establish the hyperbolicity of the essential part
of the linearization of the flow,
it follows that our result is in fact an alternative
proof, without using index theory,
of Theorem $1.4$ in the cited work, that we quote verbatim.
In its statement $D^2(U|_\calE)$ denotes the Hessian
of the restriction of the potential $U$ to the unit sphere
$\calE =\set{x\in E^N\mid \norm{x}=1}$.

\begin{theorem*}[Hu and Ou \cite{HuOu}, 2016]
If a planar central configuration $a_0$ is a strong minimizer,
that is, such that all the nontrivial eigenvalues of $D^2(U|_\calE)(a_0)$
are bigger than $U(a_0)$, then the essential part of the
Meyer-Schmidt decomposition of any elliptic Lagrangian solution
by $a_0$ is hyperbolic.
\end{theorem*}

\section{The mass inner product and the Hessian endomorphism}

Although almost everything we will prove here remains valid for
a wide range of interaction potentials, in order to fix ideas
let us start by considering the case of homogeneous potentials.
Let $E$ be a Euclidean space,
and let $m_1,\dots, m_N>0$ be the masses of $N$
punctual bodies, which evolve in the space $E$
under the mutual attraction determined by the
homogeneous potential
\begin{equation}\label{eq:U.homogeneous}
U(x)=\;\sum_{i<j}\;m_im_j\,r_{ij}^{-\kappa}.    
\end{equation}
Here $x=(r_1,\dots,r_N)\in E^N$ is the configuration
defined by the positions of the bodies,
and $r_{ij}=\norm{r_i-r_j}_E$ denotes the distance between them,
and $\kappa>0$ is a positive constant.
The case $\kappa=1$ corresponds to the
Newtonian gravitational model. According to
Newton's second law, the motion is governed by the system
of second order differential equations
\begin{equation}\label{eq:Newton}
m_i\ddot r_i=\frac{\partial U}{\partial r_i}(x)=
\sum_{j\neq i}\,m_im_j\,r_{ij}^{\;-(\kappa+2)}\,(r_j-r_i),
\;\;\;\textrm{ for }i=1,\dots,N
\end{equation}
or if we want,
\[
\ddot r_i=\frac{1}{m_i}\frac{\partial U}{\partial r_i}(x)=
\sum_{j\neq i}\,m_j\,r_{ij}^{\;-(\kappa+2)}\,(r_j-r_i),
\;\;\;\textrm{ for }i=1,\dots,N.
\]
This motivates the following definition.

\begin{definition}[The mass inner product]
\label{def:mass.inner.product}
Given two configurations $x=(r_1,\dots,r_N)$ and $y=(s_1,\dots,s_N)$,
their mass inner product in $E^N$ is
\[
\inner{x}{y}=m_1\inner{r_1}{s_1}_E+\dots+m_N\inner{r_N}{s_N}_E.
\]
\end{definition}

\begin{remark}
With respect to the mass inner product, the gradient
of $U$ writes
\[
\nabla U(x)=
\left(\frac{1}{m_1}\frac{\partial U}{\partial r_1}(x),\;\dots\;,
\frac{1}{m_N}\frac{\partial U}{\partial r_N}(x)
\right)
\]
and Newton's equations simply become
\begin{equation}
\label{eq:Newton.nabla}
\ddot x=\nabla U(x).
\end{equation}
\end{remark}

It is clear that Newton's law defines
an analytic local (non complete) flow in
the phase space, that is, the tangent space
of the open and dense set of configurations
without collisions. More precisely, motions take place
in the configuration space
\[
\Omega=\set{x\in E^N\mid U(x)<+\infty}
\]
and the phase space is
$T\Omega=\Omega\times E^N$.

\begin{notation}
As usual we will denote $\varphi^t$ this local
flow. More precisely, for any $(x_0,v_0)\in\Omega\times E^N$,
we write $\varphi^t(x_0,v_0)$ to designate the pair
$(x(t),\dot x(t))$, where $x(t)$ describes the
maximal solution of Newton's equation (\ref{eq:Newton.nabla})
for the initial conditions
$x(0)=x_0$ and $\dot x(0)=v_0$.  
\end{notation}

\subsection{Reduction of the center of mass}
\label{ssec:reduction.cm}
Let us explain now how the reduction due to the conservation
of linear momentum can be done.
Since the sum of the forces is zero,
the center of mass has a uniform rectilinear motion.
In our case, this is expressed as follows.
Let $G:E^N\to E$ be the linear map giving the center of mass
of a given configuration, that is, the unique vector $G(x)\in E$ such that
\[
(m_1+\dots +m_N)\,G(x)=m_1r_1+\dots+ m_Nr_N\,.
\]
Clearly, if $x(t)$ is a solution of Newton's equations (\ref{eq:Newton}),
and we write $c(t)$ for $G(x(t))$, then we have
that $\ddot c=0$, hence $c(t)=v_0\,t+c_0$ for some vectors $c_0, v_0\in E$.
More precisely, we have that $c_0=G(x(0))$ is the center of mass of
the initial configuration, and $v_0=\dot c(t)=G(\dot x(t))$,
\emph{the linear momentum of the motion}, is the center of mass
of the configuration of velocities, which is constant.

Therefore, assuming that the center of mass $c(t)$
of a given motion $x(t)$ is known, it is enough to understand
the dynamics of the relative positions
\[
s_i(t)=r_i(t)-c(t)
\]
or, if we want, the relative configuration $y(t)=x(t)-\delta(c(t))$,
where $\delta:E\to E^N$ is the linear map $\delta(r)=(r,\dots,r)$
for all $r\in E$.

\begin{definition}[Configurations of total collision]
\label{def:total.coll}
The space of configurations in which all positions coincide is denoted
\[
\Delta=\Ima \delta \subset E^N\,.
\]
\end{definition}

We observe that $\dim \Delta = \dim E$, and that $\Delta$
does not depend on the masses of the bodies.
However, its orthogonal complement does depend
on the masses, and it is precisely the subspace
of configurations with center of mass at the origin.

\begin{definition}[Centered configurations]
\label{def:centered.conf}
The space of configurations with center of mass at
the origin is denoted
\[
E^N_0=\Delta^\perp=\ker G\,.
\]
\end{definition}

Given any motion $x(t)$, the relative motion
with respect to the center of mass $c(t)$ is
$y(t)=x(t)-\delta(c(t))$. Thus, $y(t)$ is another motion
which has constant center of mass at the origin.
Finally it follows that every motion is of the form
$x(t)=y(t)+\delta(c_0+tv_0)$ where $y(t)$
is a motion in the \emph{invariant} space $E^N_0$.

For this reason, in almost all the literature,
this reduction is naturally proposed at the beginning.
However, for the purposes of this work,
we prefer to keep in mind the full space,
because the reduction will provide us with one of the
examples of decomposition for the Jacobi equation.

\subsection{The second derivative and the Hessian}
\label{ssec:second.derivative.HU}

Let us recall that in the general setting
of a smooth manifold, only the first differential $dU$
of a smooth function is well defined.
In order to define a symmetric two-tensor that fulfills
the role of the second derivative
we need to use a parallel transport.
More precisely, given two tangent vectors $u,v$
at some point $x$, we have to derive
\[
\eval{\frac{d}{dt}}{t=0}\left(d_{x(t)}U(v(t))\right)
\]
where $x(t)$ is such that $x(0)=x$, $\dot x(0)=u$,
and $v(t)$ is a parallel transport of $v=v(0)$
along $x(t)$.
If a Riemannian metric is given,
then $dU$ can be represented by the gradient vector $\nabla U$,
and we can define the Hessian endomorphism by
\[
HU_x(u)=\nabla_u\nabla U\;(x)\,.
\]
Then the Hessian tensor defined by
\[
\hess U_x(u,v)=\inner{HU_x(u)}{v}
\]
becomes symmetric, because of the symmetry and compatibility
of the connection with the metric.
Also, this two-tensor corresponds to the second derivative
for the Riemannian connection since, taking
$v(t)$ parallel along $x(t)$ as before, we have
\[
\inner{HU_x(u)}{v}=
\inner{D\nabla U(0)}{v}=
\eval{\frac{d}{dt}}{t=0}\inner{\nabla U(x(t))}{v(t)}
\]
where capital $D$ denotes the covariant derivative.
As we have said, we will consider the space of configurations
$E^N$ endowed with the mass inner product,
which can be viewed as a Riemannian metric.
Since the mass product is a constant metric,
parallel vectors are constant vectors, thus actually we have
\[
\hess U_x(u,v)=
\eval{\frac{d}{dt}}{t=0}\,dU_{x+tu}\;(v)=
\eval{\frac{d}{dt}}{t=0}\eval{\frac{d}{ds}}{s=0}
U(x+tu+sv)
\]
for any $(u,v)\in E^N\times E^N$.
Therefore, in terms of matrices we can write
\[
\hess U_x(u,v)=u\,D^2U_x\,v^{\,T}
\]
where $D^2U_x$ is the symmetric matrix
\[
D^2U_x=
\left(
\frac{\partial^2U}{\partial r_i\partial r_j}(x)
\right)_{i\,,\,j\,=\,1,\,...\,,\,N}
\]
whose entries are in turn the symmetric matrices 
\[
\frac{\partial^2U}{\partial r_i\partial r_j}(x)=
\left(
\frac{\partial^2U}{\partial r_i^{\,k}\,\partial r_j^{\,l}}(x)
\right)_{k\,,\,l\,=\,1,\,...\,,\,d}
\]
being $d=\dim E$.

Our point here is just that, as well as the gradient
with respect to the mass inner product turns out to be
the most convenient and natural for analyzing
the equations of motion,
for study the first variations it is natural
and convenient to consider the Hessian endomorphism
with respect to the mass inner product.
We hope that what follows will convince the reader of this statement.

\begin{notation}
If $U$ is a smooth function on a vector space,
let us denote $d^2U_x$ the usual second derivative
of $U$. Of course, this symmetric two-tensor is
the common Hessian of all constant metrics on the vector space.
\end{notation}

\begin{definition}
[The Hessian for the mass inner product]
\label{def:Hessian}
For any configuration without collisions $x\in\Omega$,
the Hessian endomorphism of $U$ is the unique operator
\[
HU_x:E^N\to E^N
\]
representing the second derivative of the Newtonian potential
at configuration $x$ 
with respect to the mass inner product, that is to say,
such that for all $u,v\in E^N$
\[
d^2U_x(u,v)=\inner{HU_x(u)}{v}\,.
\]
\end{definition}

The Hessian $HU_x$ is obviously a self-adjoint operator
for the mass inner product in $E^N$.
Moreover, it follows from the definition and
previous observations that
\begin{equation}
\label{eq:Hessian.formula}
HU_x(u)=\eval{\frac{d}{dt}}{t=0}\nabla U(\gamma(t))    
\end{equation}
where $\gamma(t)$ is any curve such that $\gamma(0)=x$
and $\dot\gamma(0)=u$.
In terms of the above defined matrices,
using a canonical basis,
as well a diagonal matrix $M$,
we can identify this endomorphism with the matrix
\begin{equation}
\label{eq:matrix.HU}
\textrm{HU}_x=\left(\frac{1}{m_i}
\frac{\partial^2U}{\partial r_i\partial r_j}(x)
\right)_{i\,,\,j\,=\,1,\,...\,,\,N}\,=
M^{-1}D^2U_x
\end{equation}

\subsection{The Jacobi equation}
\label{ssec:Jacobi.eq}
Let $x$ be a solution of $\ddot x=\nabla U(x)$
for $t\in (a,b)$, and let also a smooth family
$x_s(t)$ of such solutions,
$s\in (-\epsilon,\epsilon)$ and such that $x_0=x$.
To this variation, there is an associated vector field,
namely the field
\begin{equation}\label{eq:Jacobi.field}
J:(a,b)\to E^N\;,\quad
J(t)=\eval{\frac{d}{ds}}{s=0}\,x_s(t)\,.    
\end{equation}
It is easy to see that this kind of vector fields along
the motion $x(t)$ are exactly the solutions
to the \emph{Jacobi equation}
\begin{equation}\label{eq:Jacobi.equation}
\ddot J=HU_{x(t)}(J),    
\end{equation}
which has unique solution for initial conditions
$J(t_0)=J_0$ and $\dot J(t_0)=\dot J_0$.
Since our space has linear structure, for
these initial conditions we can define the
variation as
$x_s(t)=\pi\circ
\varphi^t(\,x(t_0)+s\,J_0\,,\,\dot x(t_0)+s\,\dot J_0)$,
where $\pi:\Omega\times E^N\to\Omega$ returns the first component.
On the other hand, given $J(t)$ defined as
in (\ref{eq:Jacobi.field}), one has
\[
\ddot J(t)=
\frac{d^2}{dt^2}\eval{\frac{d}{ds}}{s=0}x_s(t)=
\eval{\frac{d}{ds}}{s=0}\nabla U(x_s(t))=
HU_{x(t)}\,(J(t))
\]

\begin{definition}
Solutions of a Jacobi equation are called Jacobi fields.
\end{definition}

\begin{remark}
Given any motion $x(t)$ the velocity field $J(t)=\dot x(t)$
is the trivial Jacobi field, associated to the variation
$x_s(t)=x(t+s)$.
\end{remark}
\vs

\section{The splitting lemma for the Jacobi equation}

In all the cases we will consider, the decomposition
will always result from the application of the following
splitting lemma, which, despite being elementary,
we have not seen stated in the literature.
For this reason, we think it is worth giving
an abstract formulation for it.

\begin{definition}[Invariant subspace]
\label{def:invariant.subspace}
Let $O\subset E$ be an open subset of a Euclidean space,
and let $U:O\to\R$ be a $C^2$ function.
We say that a subspace $V\subset E$ is an invariant subspace
for $\ddot x=\nabla U(x)$ if we have
$\nabla U(x)\in V$ for all $x\in O\cap V$. 
\end{definition}

Equivalently,
$V\subset E$ is an invariant subspace for $\ddot x=\nabla U(x)$
if and only if any motion $x(t)$ with initial conditions
$x(t_0),\dot x(t_0)\in V$ is contained in $V$.
The following is a more precise formulation of Lemma
\ref{lemma:Jacobi.splitting} stated in the introduction.

\begin{lemma}[Splitting lemma]
\label{lemma:splitting}
Let $O\subset E$ be an open subset of a Euclidean space
and $U:O\to\R$ a $C^2$ function. If $V\subset E$
is an invariant subspace for $\ddot x=\nabla U(x)$, and $W=V^\perp$,
then for any configuration $x\in V$ we have that
\[
HU_x(V)\subset V\;,
\quad\textrm{and}\quad
HU_x(W)\subset W\,.
\]
In particular, given a Jacobi field $J(t)$
along a motion $x(t)$ in $V$,
it decomposes as $J(t)=J_V(t)+J_W(t)$, where the terms
are respectively solutions of
\[
\ddot J=\eval{(HU_x)}{V}(J)\;,
\quad\textrm{and}\quad
\ddot J=\eval{(HU_x)}{W}(J)\,.
\]
\end{lemma}

\begin{proof}
The proof is trivial. If $x\in O\cap V$, and $v\in V$,
then we have that
\[
HU_x(v)=\eval{\frac{d}{dt}}{t=0}\,\nabla U(x+tv)\in V
\]
and we conclude that $HU_x(V)\subset V$.
Since $HU_x$ is self-adjoint for the inner product of $E$,
we also conclude that $HU_x(W)\subset W$.
Therefore,
if we decompose the Jacobi field as $J(t)=J_V(t)+J_W(t)$,
we can write
\[
\ddot J=\ddot J_V+\ddot J_W
\]
\[
HU_{x(t)}(J)=HU_{x(t)}(J_V)+HU_{x(t)}(J_W)
\]
hence the proof follows from the uniqueness
of the decomposition in $E=V\oplus W$.
\end{proof}

\section{Applications to the {\it N}-body problem}
\label{sec:applications.list}

We give now four examples in the case
of the $N$-body problem for which our splitting lemma applies.
Although the first two cases are essentially trivial,
they allow us to complete the description
We are particularly interested in the third example,
which applies to homographic orbits,
because from it one can deduce the well-known
Meyer-Schmidt symplectic decomposition
\cite[Proposition 2.1, p.260]{MeyerSchmidt}.

\begin{enumerate}
    \item[(i)]
    \emph{Reduction of the center of mass}.
    This is the simplest example we can consider:
    Take $V=E^N_0$, then $V^\perp=\Delta$ is the
    space of configurations of total collision.
    Clearly we have $\nabla U(x)\in E^N_0$ for all $x\in\Omega$,
    therefore the splitting lemma applies.
    If $x(t)$ is a motion with fixed center of mass,
    then any Jacobi field decomposes as
    \[
    J(t)=J_0(t)+J_\Delta(t)
    \]
    where $\ddot J_0=HU_0(J)$, $HU_0=\eval{(HU_x)}{E^N_0}$,
    and $\ddot J_\Delta=0$ since $\Delta\subset\ker HU_x$
    for any $x\in\Omega$.
    Of course, this also follows immediately from the variations
    we can make to the form $x_s(t)=x(t)+t\delta(sv)+\delta(sc)$
    for any $v,c\in E$, see the reduction made in section
    \ref{ssec:reduction.cm}.\vs
    
    \item[(ii)]
    \emph{Coplanar configurations}.
    Suppose the $N$ bodies move in a three-dimensional space $E=\R^3$.
    If we fix a plane $S\subset E$, we can consider motions in which
    all bodies remain in $S$. In this case, or more generally if
    we consider any fixed $k$-dimensional subspace $S$ of a
    $d$-dimensional Euclidean space $E$, we have that $V=S^N\subset E^N$
    is an invariant subspace with orthogonal $W=(S^\perp)^N$.\vs
    
    \item[(iii)]
    \emph{Planar central configurations as Keplerian embeddings}.
    Let us recall that a central configuration is a configuration
    admitting homothetic motions. In other words, $x_0$ is a
    central configuration if there are motions of the form
    $x(t)=\varphi(t)\,x_0$,
    being $\varphi$ a positive real valued function\footnote{\,
    This definition implies that central configurations
    must be centered at the origin.
    Some authors allow the homotheties
    to be with respect to the center of mass of the configuration,
    thus in this sense, they allow any translation of a central
    configuration to be also central.}.
    Recalling that the potential $U$
    is homogeneous of degree $-\kappa<0$,
    it turns out that such a curve is a true motion,
    if and only if
    \begin{equation}\label{eq:cc.nablaU//x}
    \nabla U(x_0)=\lambda\,x_0\;\quad\textrm{ for some }\lambda\in\R
    \end{equation}
    and $\varphi$ satisfies
    \begin{equation}
    \ddot \varphi(t)\,x_0=
    \nabla U(\varphi(t)\,x_0)=\varphi(t)^{-(\kappa+1)}\lambda\,x_0.
    \end{equation}
    \vs
    From equation (\ref{eq:cc.nablaU//x}) and Euler's theorem
    we get that
    \[
    \lambda=-\kappa \,\frac{U(x_0)}{I(x_0)}\,,
    \quad\textrm{ where }\quad
    I(x_0)=\norm{x_0}^2
    \]
    is the moment of inertia of the configuration
    with respect to the origin.
    We also deduce that $\varphi(t)$ must solve the one-dimensional
    equation in $\R$
    \begin{equation}\label{eq:central.force.kappa}
    \ddot \varphi=\lambda\,\varphi^{-(\kappa+1)}.    
    \end{equation}
    Since we are considering the planar case, that is $\dim E=2$,
    we can make a natural identification of $E^N\simeq \C^N$,
    thus providing an action of $\C$ on the space of configurations.
    As is known, it turns out that solutions to
    the central force problem (\ref{eq:central.force.kappa}) in $\C$,
    produce homographic solutions of (\ref{eq:Newton}).
    When $\kappa=1$ the central force problem is nothing but
    the Kepler problem in the plane, and solutions are Keplerian
    ellipses, parabolas or hyperbolas.
    It is clear that if we define the subspace $K=\C\,x_0$,
    then $\nabla U$ is a central field on $K$,
    which implies that $K$ is invariant.
    Therefore our lemma applies and
    $HU_{x_0}$ splits for the orthogonal sum $K\oplus K^\perp$.
    \vs

    \item[(iv)]
    \emph{Isosceles configurations}.
    This is a well-known example of a subproblem
    of the three-body problem when two masses are equal.
    Without loss of generality we can assume that the value
    of the two equal masses is $1$, and we denote $m>0$
    the value of the third mass.
    Assume that $\dim E=3$, and let $L\subset E$ be a
    one-dimensional subspace.
    Let $r_1$ and $r_2$ be the positions of the two equal masses,
    and call $r_3$ the position of the mass $m$.
    Then we define the space of isosceles configurations
    \[
    I=\set{(r_1,r_2,r_3)\in E^3_0\textrm{ such that }
    \;r_3\in L\textrm{ and }\;r_2-r_1\in L^\perp}.
    \]
    It is not difficult to verify that $I$ is an invariant
    subspace for the classical three-body problem, as well
    as that for any potential defining attractive forces
    only depending on the distances.
    
    In this case, the application of the splitting lemma gives
    for any $x\in I\setminus\Delta$ a splitting of $HU_x$
    with respect to the orthogonal sum
    $E^3=\Delta\oplus I\oplus D$ where $I\oplus D=E^3_0$.
    Note that each one of the terms has dimension $3$.
\end{enumerate}
\vs

\subsection{The Meyer-Schmidt decomposition}
We will see now how the well-known Meyer-Schmidt decomposition
can easily be deduced.
Recall that it deals with the case of a planar central
configurations of the Newtonian $N$-body problem and their
homographic elliptic motions.
Although these authors carry out the decomposition
for the Newtonian case, we can see that our point of view also
applies to any homogeneous potential, or even for the logarithmic one.
Regarding these potentials, the stability of their circular
equilateral relative equilibria was studied by
Barutello, Jadanza and Portaluri, who adapted the
construction of symplectic coordinates in the phase space
specifically for these motions, see
\cite[p.10--11]{BarutelloJadanzaPortaluri}.

Let $x_0$ be a planar central configuration, and let
$K=\C\,x_0$ the space of all central configurations,
similar to $x_0$, and having the same orientation.
The splitting lemma gives then\footnote{\,Recall that
$\Delta$ is the space of configurations of total collisions,
composed by configurations of the form $x=(r,\dots,r)$,
and that $\Delta^\perp=E_0^N$ is the space of configurations
with center of mass at the origin of $E$ (see definitions
\ref{def:total.coll} and \ref{def:centered.conf}).
Note that $K$ is composed by central configurations,
so we have $K\subset E_0^N$. Since $K$ is $HU$ invariant,
as well as $\Delta$ (see item (i) in the previous section),
we deduce the invariance of $D=(\Delta+K)^\perp$, that is
the orthogonal of $K$ inside $E_0^N$.}
the orthogonal decomposition
for the Hessian endomorphism $HU_x$,
for any configuration $x\in K$,
\[
E^N=\Delta\oplus K\oplus D\,.
\]
Since $\dim \Delta=\dim K =2$, we have $\dim D=2(N-2)$.
Therefore, given any homographic motion $x(t)$ in $K$,
his Jacobi fields decompose as
\[J=J_\Delta+J_K+J_D\,.\]

Observing that the linearization of the Lagrangian flow
in the phase space is nothing but the first order
system of differential equations
\begin{equation}\label{eq:linear.E-L}
\begin{array}{rcl}
    \dot J&=&Z\\
    \dot Z&=&HU_x(J)
\end{array}
\end{equation}
we conclude that the spaces $\calA_0=\Delta\times\Delta$,
$\calB_0=K\times K$ and $\calC_0=D\times D$ are invariant
for the linearized Lagrangian flow. Note that these spaces
have dimension $4$, $4$ and $4(N-2)$ respectively.
Taking finally the Legendre transform of these three subspaces
we get the three subspaces $\calA$, $\calB$ and $\calC$
in the cotangent space, which are invariant for the linearization
of the Hamiltonian flow, and to which the Meyer-Schmidt
symplectic coordinates are associated.
More precisely, The Legendre transform is
the diffeomorphism $\calL:T\Omega\to T^*\Omega$
\[
\calL(x,v)=
\left(x,L_v(x,v)\right)
\]
where
\[
L(x,v)=\frac{1}{2}\norm{v}^2+U(x)
\]
is the Lagrangian of the system. Recalling
that the Legendre transform conjugates the
Lagrangian with the Hamiltonian flow,
we conclude that their linearizations are
conjugated by the differential $d\calL$.
It follows, since in our case the Lagrangian is quadratic,
that the Legendre transform is linear, and so we have $d\calL=\calL$.
The conjugation is then performed by $\calL$.

Using the identification of $T_x^*\Omega=(E^*)^N$
we get the classical expression
\[
\calL(x_1,\dots,x_N,v_1,\dots,v_N)=
(x_1,\dots,x_N,m_1v_1,\dots,m_Nv_N)
\]
for the Legendre transform, which explains the description
of the Meyer-Schmidt coordinates in
\cite[p. 259]{MeyerSchmidt}, that is to say,
\begin{eqnarray*}
\calA&=&\calL(\Delta\times\Delta)=
\set{(\,b\,,\dots,\,b\,,\,m_1c\,,\dots,\,m_Nc\,)\mid b,c\in\R^2},\\    
\calB&=&\calL(K\times K)=
\set{(\,za_1\,,\dots,\,za_N\,,\,m_1w\,a_1\,,\,\dots\,,\,m_Nw\,a_N\,)
\mid z,w\in\C},
\end{eqnarray*}
where $x_0=(a_1,\dots,a_N)$ is the central configuration
so that $K=\C\,.\,x_0$, and
\[
\calC=\calL(D\times D)=
\set{
(\,d_1\,,\,\dots\,,\,d_N\,,\,m_1d'_1\,,\,\dots\,,\,m_Nd'_N\,)
\mid d=(\,d_i\,),\,d'=(\,d'_i\,)\in D}.
\]
Proposition $2.1$ and Theorem $2.1$ in the paper of
Meyer and Schmidt \cite{MeyerSchmidt}
follows immediately from the orthogonality of our
subspaces $\Delta$, $K$ and $D$.
\vs

\section{Instability of elliptic homographic motions}

We will use now our splitting lemma for the Jacobi
equation, more precisely application (iii) of previous section,
in order to prove Theorem \ref{thm:StrongND->unstable}.
In fact, we are going to show that in one of the components,
the one corresponding to the space $D=(K+\Delta)^\perp$,
where $K=\C\,x_0$, all non-trivial solutions
have exponential growth in the future or in the past.
This is equivalent to the normal hyperbolicity
for the linearization of the Lagrangian flow
given by (\ref{eq:linear.E-L}),
implying the linear instability of the motion.

\begin{notation}
In what follows, given $z\in\C$ and $x=(r_1\dots,r_N)\in E^N$,
we will write
\[
z\,x =(z\,r_1\,,\dots,\,z\,r_N).
\]
\end{notation}

\begin{remark}
\label{remark:<zx,zy>=z2<x,y>}
For any two configurations $x,y\in E^N$ and $z\in\C$
we have that
\[
\inner{z\,x}{z\,y}=\abs{z}^2\inner{x}{y}.
\]
\end{remark}

We will use the following lemma, which is due to
the homogeneity and the rotational invariance
of the Newtonian potential.

\begin{lemma}\label{lemma:homogenityHU}
    Let $x\in\Omega$ be any planar configuration, and $z\in\C$
    such that $z\neq 0$. Then for any $v\in E^N$ we have
    \[
    HU_{z\,x}\,(v)=z\,\abs{z}^{-3}HU_x(z^{-1}v).
    \]
\end{lemma}

\begin{proof}
For any configuration $x\in\Omega$ and $z\neq 0$ we have $U(z\,x)=\abs{z}^{-1}U(x)$.
Thus the chain rule gives
\[
dU_{z\,x}(v)=\abs{z}^{-1}dU_x(z^{-1}v),
\]
hence, using the definition of the gradient and
Remark \ref{remark:<zx,zy>=z2<x,y>} above, we have
\begin{eqnarray*}
\inner{\nabla U(z\,x)}{v}&=&\abs{z}^{-1}\inner{\nabla U(x)}{z^{-1}v}\\
&&\\
&=&\,\abs{z}^{-3}\inner{z\,\nabla U(x)}{v}
\end{eqnarray*}
for any $v\in E^N$, therefore
\[
\nabla U(z\,x)=z\abs{z}^{-3}\nabla U(x).
\]
The proof is then achieved by using formula \ref{eq:Hessian.formula},
since given $u\in E^N$ we can write
\begin{eqnarray*}
HU_{z\,x}(u)&=&
\eval{\frac{d}{ds}}{s=0}\;\nabla U(z\,x+su)\\
&&\\
&=&\;z\,\abs{z}^{-3}\;\eval{\frac{d}{ds}}{s=0}\;
\nabla U(x+sz^{-1}u)\\
&&\\
&=&\;z\,\abs{z}^{-3}HU_x(z^{-1}u)\,.
\end{eqnarray*}
\end{proof}

The previous lemma allows us to immediately deduce the following.

\begin{lemma}
Let $x_0$ be a planar central configuration, let $K=\C\,x_0$
and let
\[
D=(K+\Delta)^\perp=K^\perp\cap E^N_0.
\]
If $z(t)$ is a solution to the Kepler
problem in the plane
\[
\ddot z=-\frac{U(x_0)}{I(x_0)}\;z^{-2}
\]
and $x(t)=z(t)\,x_0$ the corresponding homographic solution
of the planar $N$-body problem, then writing
$A_t=HU_{x(t)}$ we have that
\[
\inner{A_t(v)}{v}=\,\abs{z(t)}^{-1}\inner{A_0(\,z(t)^{-1}v\,)}{z(t)^{-1}v}
\]
for all $t\in\R$ and any $v\in D$.
\end{lemma}

\begin{proof}
By the previous lemma we have
\begin{eqnarray*}
\inner{A_t(v)}{v}&=&\inner{HU_{z(t)\,x_0}(v)}{v}=
\abs{z(t)}^{-3}\;\inner{\,z(t)\,HU_{x_0}(z(t)^{-1}v)}{v}\\
&&\\
&=&\abs{z(t)}^{-1}\;\inner{HU_{x_0}(z(t)^{-1}v)}{z(t)^{-1}v}\\
&&\\
&=&\abs{z(t)}^{-1}\inner{A_0(\,z(t)^{-1}v\,)}{z(t)^{-1}v}
\end{eqnarray*}
\end{proof}

In particular, if the solution $z(t)$ is an ellipse,
we get a positive lower bound for the quadratic
form defined by $HU$ in $D$ along the homographic solution.
We recall that our definition \ref{def:cc.strong.ND}, of
strongly non-degenerate central configuration
requires precisely that the restriction of the Hessian
to the subspace $D$ be non-degenerate, since we have
$S_{x_0}=K+\Delta$, so $S_{x_0}^\perp=D$.
In other words,
$S_{x_0}$ is the space of configurations obtained from $x_0$
by rotations, homotheties and translations

\begin{lemma}\label{lema:HUx(t)>alpha.x^2}
In the same hypothesis of the previous lemma,
if $z(t)$ is a Keplerian ellipse and we know that
$x_0$ is strongly nondegenerate, see definition \ref{def:cc.strong.ND},
then for any $t\in\R$ and any $v\in D$ we have
\[
\inner{A_t(v)}{v}\geq \;k^{-3}\mu\,\norm{v}^2
\]
where
\[
k=\max_{t\in\R}\;\abs{z(t)}\;,
\quad \textrm{ and }\quad
\mu=\min_{\norm{u}=1}\;\inner{A_0(u)}{u}\,.
\]
\end{lemma}

Now we prove the following comparison theorem,
which is mainly inspired in Rauch's Theorem
\cite[Chapter 10, p. 215]{doCarmo}.
We will compare solutions of two ordinary differential equations
in $\R^n$, namely
\begin{equation}
\label{eq:Hill.equ}
\textrm{(1) }\ddot x=A(t)\,x\;,
\quad\textrm{ and }\quad
\textrm{(2) }\ddot x=\alpha\,x\;,
\end{equation}
where $\alpha>0$, and $A(t)$ is a periodic and positive definite
matrix satisfying
\begin{equation}
\label{eq:A(t).positive}
\inner{A(t)x}{x}\geq \alpha \norm{x}^2
\end{equation}
for all $x\in\R^n$ and all $t\in\R$.

\begin{theorem}[Comparison]\label{thm:comparison}
Let $A(t)\in\calM_n(\R)$ be a matrix satisfying
(\ref{eq:A(t).positive}) for some $\alpha>0$ and every $t>0$.
Let also $x(t)$ be a non trivial solution of $\ddot x=A(t)x$
such that $x(0)=0$, and let $z(t)$ be the solution of
$\ddot x=\alpha x$ with $z(0)=0$ and
$\dot z(0)=\dot x(0)\neq 0$. Then, for all $t>0$ we have
\[
\norm{x(t)}\geq \norm{z(t)}.
\]
\end{theorem}

It is worth noting, as an interesting anecdote, that
periodic linear equations as the first one in (\ref{eq:Hill.equ})
are known as Hill's equations since they were first introduced
by George W. Hill in order to study the variations of
the Lunar perigee \cite{Hill}.

Despite the very basic character of Theorem \ref{thm:comparison},
we are not able to point a reference,
and for this reason, we provide the following one below.
Then we will be able to prove our main application, that is, that
any elliptic homographic motion by a strongly non-degenerate
central configurations is unstable.
The proof uses a couple of lemmas, for which we need
to introduce the corresponding Lagrangian functions
of systems (1) and (2) in \ref{eq:Hill.equ} above, that is,
\[
L(x,v,t)=\norm{v}^2+\inner{x}{A(t)x}
\quad\textrm{ and }\quad
L_0(x,v)=\norm{v}^2+\alpha\,\norm{x}^2
\,.
\]
The corresponding action functionals,
evaluated on a given curve $w\in C^1([0,t],\R^n)$
take the values
\[
\calA_0(w)=\int_0^t \norm{\dot w(s)}^2+\alpha\,\norm{w(s)}^2\;ds
\]
and
\[
\calA(w)=\int_0^t \norm{\dot w(s)}^2+\inner{x(s)}{A(s)x(s)}\;ds\,.
\]
We prove first that the solutions $z(t)$ of system (2)
such that $z(0)=0$ are minimizing of $\calA_0$
among the curves with the same endpoints.
Note that these solutions are of the form $z=r(t)\,u$,
for some $u\in\R^n$, and $r(t)$ a solution of $\ddot r=\alpha\,r$.

\begin{lemma}\label{lema:z.minimiza}
Let $z(s)$ be a solution of $\ddot x=\alpha x$ defined for
$s\in [0,t]$ and such that $z(0)=0$.If $w\in C^1([0,t],\R^n)$
is such that $w(0)=z(0)$ and $w(t)=z(t)$, then
\[
\calA_0(z)\leq \calA_0(w).
\]
\end{lemma}

\begin{proof}
Let $z(s)=r(s)u$ be the solution defined in $[0,t]$
starting at $z(0)=0$,
were we assume that $\norm{u}=1$.
We decompose the curve $w$ as the sum
of a collinear to $z$ component and an orthogonal one.
Thus we write $w(s)=\eta(s)u+v(s)$, where $v(s)\in u^\perp$
for all $s\in [0,t]$. We have $\dot w(s)=\dot\eta(s)u+\dot v(s)$,
so in particular $\norm{\dot w(s)}^2\geq \dot\eta(s)^2$ as well as
$\norm{w(s)}^2\geq\eta(s)^2$. We deduce that
\[
\calA_0(w)\geq
\int_0^t\dot\eta(s)^2+\alpha\,\eta(s)^2\,ds\,
=a_0(\eta)
\]
where $a_0$ is the Lagrangian action for the one-dimensional
equation $\ddot r=\alpha\,r$. It is clear that the functional
$a_0$ is coercive, and that the minimization problem
has a unique solution once we fix the endpoints. Therefore
we have that
\[
a_0(\eta)\geq a_0(r)=\calA_0(z)\,.
\]
\end{proof}

\begin{lemma}\label{lema:derivada.norm.cua}
Let $x(t)$ be a solution of $\ddot x=A(t)x$
such that $x(0)=0$, and define the function
$h(t)=\norm{x(t)}^2$. Then, for each $t>0$ we have
$h'(t)=2\,\calA(x_t)$ where $x_t$ denotes the
restriction of the solution to the interval $[0,t]$.
\end{lemma}

\begin{proof} We have
\begin{eqnarray*}
h'(t)&=&\frac{d}{dt}\inner{x(t)}{x(t)}=
2\inner{x(t)}{\dot x(t)}\\
&=& 2\int_0^t\frac{d}{ds}\inner{x(s)}{\dot x(s)}ds\\
&=& 2\int_0^t
\norm{\dot x(s)}^2+\inner{x(s)}{A(s)x(s)}\,ds
\;=\;2\,\calA(x_t)\,.
\end{eqnarray*}
\end{proof}

We can prove now the Rauch type comparison theorem.

\begin{proof}[Proof of Theorem \ref{thm:comparison}]
We have a solution $x(t)$ of $\ddot x=A(t)x$, and a solution
$z(t)$ of $\ddot x=\alpha\,x$, satisfying $x(0)=z(0)=0$ and
$\dot x(0)=\dot z(0)\neq 0$. We define the functions
$h(t)=\norm{x(t)}^2$ and $g(t)=\norm{z(t)}^2$.
Since they are positive for $t>0$, we can write
$h(t)=\lambda(t)^2g(t)$ for a positive function $\lambda(t)$.

We can assume that $z(t)=\rho(t)u$, with $\norm{u}=1$,
and for each $t>0$, there is a unique $v(t)\in\R^n$,
with $\norm{v(t)}=1$, such that $y_t(s)=\lambda(t)\rho(s)v(t)$
is a solution of $\ddot x=\alpha\,x$
which satisfies $y_t(t)=x(t)$.
since the functional $\calA_0$ is homogeneous,
also that $\calA_0(y_t)=\lambda(t)^2\calA_0(z_t)$, where $z_t$ is the restriction
of $z$ to the interval $[0,t]$.

Let us also denote $x_t$ the restriction of $x$ to the interval $[0,t]$.
Because of the previous Lemmas \ref{lema:z.minimiza} and \ref{lema:derivada.norm.cua},
as well the fact that $\calA\geq \calA_0$, we have
\[
h'(t)=2\calA(x_t)\geq 2\calA_0(x_t)\geq 2\calA_0(y_t).
\]
Now, applying Lemma \ref{lema:derivada.norm.cua} to the matrix $A(t)=\alpha\, I$
we get
\[
2\calA_0(y_t)=2\lambda(t)^2\calA_0(z_t)=\lambda(t)^2g'(t).
\]
We conclude that for any $t>0$ the inequality $h'(t)\geq \lambda(t)^2 g'(t)$ holds.
Since we know that $h(t)=\lambda(t)^2g(t)$, we conclude that the function
$f(t)=h(t)g(t)^{-1}$, defined for $t>0$, has positive derivative. On the other hand,
by L'Hôpital's rule we know that $\lim_{t\to 0}f(t)=1$. Therefore we
have proved that $f(t)\geq 1$ for all $t>0$, which says that
\[
\norm{x(t)}\geq \norm{z(t)}
\]
for all $t>0$ as we wanted to prove.
\end{proof}

\subsection{Proof of Theorem \ref{thm:StrongND->unstable}}

We will show that given a periodic homographic motion
of the planar $N$-body problem by a strongly non-degenerate
configuration is linear unstable.
Actually, as Ou in \cite{Ou},
we prove first that the essential part of the Meyer-Schmidt
decomposition is hyperbolic

\begin{proof}
Let $x_0$ be a strongly non-degenerate central configuration,
and let $x(t)$ be a homographic periodic motion, of period $T>0$,
and such that $x(0)=x(T)=x_0$.
The linearization of the Euler-Lagrange flow over $x(t)$
is given by system (\ref{eq:linear.E-L}),
in $E^N\times E^N$ for which the space $\calD=D\times D$
is known to be invariant.
Recall that $D=(K+\Delta)^\perp$ where $K=\C\,x_0$.
Let us write the restriction
of this system to $\calD$ as the linear periodic system
\begin{equation}
\label{eq:linearEL.restricted.D}
\dot X=A(t)\,X
\end{equation}
where
\[
X=(J,V)\,,\quad
A(t)=
\left(\;
\begin{array}{cc}
     0&I  \\
     B(t)&0 
\end{array}\;
\right),
\]
and $B(t)$ is the transformation matrix of the
restriction of $HU_{x(t)}$ to subspace $D$.
Let us also call $M=\phi_0(T)$ the monodromy matrix
of this system. Then, for any Jacobi field $J$ in $D$
and any $n\in\Z$ we have
\[
(J(n),\dot J(n))=M^n(J(0),\dot J(0)).
\]
On the other hand we know, due to Theorem \ref{thm:comparison}
combined with the Lemma \ref{lema:HUx(t)>alpha.x^2},
that any non-trivial solution of equation (\ref{eq:linearEL.restricted.D})
with initial condition in
\[
\calD_0=\set{X=(0,V)\mid V\in D}
\]
has exponential growth, both in the past and in the future.
More precisely, if $X=(0,V)\in \calD_0$ and $V\neq 0$, and we
denote $X(t)=(J(t),\dot J(t))$ the solution of (\ref{eq:linearEL.restricted.D})
with initial condition $J(0)=0$ and $\dot J(0)=V$, then
$J(t)$ is a solution of the second order differential equation
$\ddot J=B(t)J$. By Lemma \ref{lema:HUx(t)>alpha.x^2}, we
know that there is $\alpha>0$ such that
\[
\inner{B(t)x}{x}\geq \alpha\norm{x}^2
\]
for all $x\in D$. This allow us to apply Theorem \ref{thm:comparison}
and conclude that, for $t>0$,
\begin{equation}
\label{eq:normJ.geq.normZ}
\norm{J(t)}\geq\norm{Z(t)}    
\end{equation}
where $Z(t)$ is the solution of $\ddot Z=\alpha Z$ with
the same initial conditions, that is
\[
Z(t)=\frac{1}{2}\,
(e^{\sqrt{\alpha}t}-e^{-\sqrt{\alpha}t})\,V.
\]
Since $J(-t)$ is the solution of $\ddot J=B(t)J$ with
initial conditions $J(0)=0$ and $\dot J(0)=-V$,
and $Z(-t)=-Z(t)$, we have in fact that that
(\ref{eq:normJ.geq.normZ}) holds for all $t\in\R$.

Therefore we conclude that, if $X\in \calD_0$ and $X\neq 0$,
then $M^nX$ has exponential growth for $n\to \pm\infty$.
Let us call $d=\dim D$. By grouping generalized eigenspaces of $M$,
we do the decomposition
\[
\calD=E^s\oplus E^c \oplus E^u
\]
such that
\begin{eqnarray*}
    \lim_{n\to +\infty}M^nX&=&0\quad\textrm{ if }X\in E^s,\\
    &&\\
    \lim_{n\to -\infty}M^nX&=&0\quad\textrm{ if }X\in E^u,
\end{eqnarray*}
and
\[
\norm{M^nX}\leq\;\rho\,\abs{n}^{2d-1}\norm{X}
    \quad\textrm{ for some }\rho>0
    \quad\textrm{ if }X\in E^c.
\]
Since $\calD_0\cap (E^c\oplus E^u)=0$, we get that
$\dim E^s\geq \dim \calD_0=d$.
We also have $\calD_0\cap (E^s\oplus E^c)=0$, hence
that $\dim E^u\geq \dim\calD_0$. Since $\dim\calD=2d$,
we conclude that $\dim E^s=\dim E^u=d$ and that $E^c=0$,
showing that $M$ is hyperbolic.

The restriction of the
linearized Euler-Lagrange equation (\ref{eq:linear.E-L})
to the invariant subspace $\calD$ is then hyperbolic.
While it is classical to deduce from this that the periodic
orbit is unstable, we give the argument below
for the sake of completeness.

We choose a Poincaré section $\Sigma$
for the periodic orbit $(x(t),\dot x(t))$ in $T\Omega$,
at the point $p=(x_0,\dot x_0)$,
in such a way that $\set{p}\times W\subset\Sigma$
where $W$ is a neighbourhood of $0$ in $\calD$.
Let us call $g$ the corresponding first recurrence map,
defined in a neighbourhood of $p$ in $\Sigma$, so $g(p)=p$.
Since $\calD$ is invariant for the linearization of the flow,
it is easy to see that for all $v\in\calD$ we have 
\[
dg_p(v)=d\varphi^T_p(v)=Mv\,,
\]
see for instance Proposition 3, Chapter 13
in \cite{HirschSmale}. The eigenvalues of $dg_p\mid_\calD$
are then of modulus greater than or less than $1$,
which shows the instability.
\end{proof}
\vs

\subsection{On the Theorem of Hu and Ou}
\label{sec:definitions.StrongND.equivalent}
As we said in the introduction,
Hu and Ou obtained in \cite[Theorem 1.4]{HuOu}
the hyperbolicity of the essential part of the linearized flow,
just as we did in the preceding proof, using index theory.
Again, the result is obtained there by means
of a hypothesis about the central configuration,
which they call \emph{strong minimization}.
We shall now see that, in fact, this definition
is completely equivalent to the one that serves
as the hypothesis for Theorem 1.2,
which we have called \emph{strong non-degeneracy}.

Let us write $E=\R^2$, and let $a\in E^N$
be a normal planar central configuration,
\[
a\in S=\set{x\in E^N_0\mid\norm{a}=1}.
\]
So we are considering only centered configurations.
Let us write, as in the reference, $D^2(U|_S)$
for the Hessian operator of the restriction of the
Newtonian potential to the unit sphere $S$.
It is clear that $U$ is constant on the circle $S\cap K$,
composed by rotations of configuration $a$.
Let us call $L\subset T_aS$ the tangent at $a$
of this circle. Thus trivially $L$ is in the kernel
of $D^2(U|_S)(a)$. A minimal configuration
is a configuration which minimizes $U|_S$,
so in particular it is a central configuration.
We say that a minimal $a$ is non-degenerate
if $\ker D^2(U|_S)(a)=L$.

\begin{definition}
\label{def:strong.minimizer}
A planar central configuration is said to be
a \emph{strong minimizer} if it is non-degenerate
and all the non-zero eigenvalues of $D^2(U|_S)$
satisfy $\lambda>U(a)$.
\end{definition}

\begin{proposition}
A planar central configuration is a strong minimizer
if and only if it is strongly non-degenerate.
\end{proposition}

\begin{proof}
Recall that $K=\C\,a$ is the space of configurations
similar to $a$, and thus it is generated by $a$ and $i\,a$.
Hence $K = \R\,a + L$ where $L = \R\,i\,a$, and
$L\subset T_aS=a^\perp$.

On the other hand $D$ is the orthogonal to $K$
inside $E^N_0$, or if we want, $D$ is the orthogonal
to $K+\Delta$ in $E^N$. Therefore we have that $T_aS=L+D$.

Now we observe that for any $x\in S$,
\begin{eqnarray*}
\nabla(U|_S)(x)&=&\nabla U(x)^T=
\nabla U(x)-\inner{\nabla U(x)}{x}\,x\\
&=&\nabla U(x)+U(x)\,x\,,
\end{eqnarray*}
the last equality being due to the fact that
the Newtonian potential is homogeneous of degree $-1$.
The vector field $Z(x)=\nabla U(x)+U(x)\,x$
is a smooth extension of $\nabla(U|_S)$
to the ambient space $\Omega$.
Therefore, the Hessian operator of $U|_S$
can be computed as the tangent component
of the derivative of $Z$ in $\Omega$. That is,
\[
D^2(U|_S)_x(v)=(\nabla_vZ)^T(x).
\]
for $x\in S$ and $v\in T_x S$.
Since $a$ is a central configuration, we have $Z(a)=0$,
hence $\nabla_vZ(a)$ is already tangent\footnote{\;
If a vector field $V$ is tangent to an
immersed submanifold $S\subset M$,
then at each point $x\in S$
the difference between its covariant derivatives,
in the ambient space and in the submanifold,
is the value of the second fundamental form
applied to $V(x)$.
Therefore the covariant derivatives agree at each point
where $V=0$.
In our case we also can see that $\nabla_vZ(a)\in T_aS$
for each $v\in T_aS$, and moreover that the
subspace $D\subset T_aS=a^\perp$ is
invariant for $D^2(U|_S)_a$ as a consequence of equality
(\ref{equa:nabla.v.Z}), Lemma \ref{lemma:splitting}
and the fact that $\nabla_vU(a)=0$ for $v\in T_aS$.}
to $S$.
On the other hand, we have
\begin{equation}
\label{equa:nabla.v.Z}
\nabla_vZ(x)=\nabla_v\nabla U(x)\;+\;\nabla_vU(x)\,x\,+\,U(x)v\,.
\end{equation}
Again, since $a$ is a central configuration, for $v\in T_aS$
we have $\nabla_vU(a)=0$, thus
\begin{equation*}
\label{equa:D2U.vs.HU}    
D^2(U|_S)_a(v)=HU_a(v)+U(a)v\,.
\end{equation*}
Therefore $a$ is a strong minimizer,
that is $D^2(U|_S)|_D > U(a)I$, if and only if
$a$ is a strongly non-degenerate central configuration,
that is, if $HU_a|_D > 0$.
\end{proof}

\section{The orthogonal to the space of equilateral triangles}

In this last section we deal with the planar three-body problem,
and in particular we compute the orthogonal space of
Lagrange equilateral configurations.

To do this we will see that in general, for planar configurations of $N$ bodies,
the mass inner product is nothing but the real part
of an \emph{Hermitian mass inner product} for which,
two plane configurations are orthogonal if and only if
the complex lines they generate are orthogonal real planes
for the usual mass inner product.

\subsection{The complex mass inner product}

We start by identifying, as before, the space of
configurations $E^N$ with the complex vector space $\C^N$.
Therefore multiplying a configuration by $i=\sqrt{-1}$
corresponds exactly to rotating it by $\pi/2$
in a direction that we will call positive.

\begin{definition}[Complex mass inner product]
\label{def:complex.mass.inner.product}
Given configurations $x=(z_1,\dots,z_N)$
and $y=(w_1,\dots,w_N)$, we define their complex product as
\[
\cinner{x}{y}\; = \inner{x}{y}+i\,\inner{x}{iy}=
m_1z_1\cj{w_1}+\dots+m_Nz_N\cj{w_N}.
\]
\end{definition}

\begin{remark}
    If the masses are equal, this is
    the standard scalar product of $\C^N$.
    With this product, we have
    $\cinner{x}{y}=0$ if and only if $\inner{x}{z\,y}=0$
    for any $z\in\C$.
\end{remark}

We know that the space $E^N_0$ of configurations with
center of mass at $0$, that is $\ker G$, is precisely
the orthogonal of $\Delta\subset E^N$ for the real mass inner product,
see definitions \ref{def:total.coll} and \ref{def:centered.conf}.
On the other hand, as a complex subspace of $E^N$, we have that
$\Delta=\left<{\bf 1}\right>$, where
${\bf 1}=\delta(1)=(1,\dots,1)$.
Therefore we have that a configuration $x\in E^N$
is centered, if and only if $\cinner{x}{{\bf 1}}\;=\;0.$

\subsection{The space of positive and centered equilateral triangles}
We consider now the case $N=3$. Let us define the configurations
\[
T=(1,\omega,\omega^2)\,,
\]
where $\omega=-(1/2)+(\sqrt{3}/2)i$ is a cube root of unity, and
\[
T_0=(1-c,\omega-c,\omega^2-c)=T-c\,{\bf 1}\,,
\]
where $c=G(T)$ is the center of mass of the equilateral configuration $T$.
The space of positive and centered equilateral triangles is therefore
\[
K=\C\,T_0
\]
and as before we define
\[
D=(K+\Delta)^\perp=
(\C\,T+\C\,{\bf 1})^\perp.
\]
Note that $\C\,T_0+\C\,{\bf 1}=\C\,T+\C\,{\bf 1}$
because $T-T_0\in\Delta$. It follows that
\[
\C^3\simeq E^3=\Delta\oplus K\oplus D
\]
and these three spaces are orthogonal complex lines.
We will call $D$ the deformation space.
In the following we will denote $q=(x,y,z)\in\C^3$
as a generic configuration, and $\C^3_0$ the
space of centered configurations, that is,
\[
\C^3_0=\set{q=(x,y,z)\in\C^3\mid m_1x+m_2y+m_3z=0}.
\]
We can say that $D$ is the orthogonal of $K$ inside $\C^3_0$.
Moreover, since $K+\Delta$ is spanned by configurations $T$ and ${\bf 1}$,
we have
\[
D=\set{q=(x,y,z)\in\C^3\mid\;\;
\cinner{q}{(1,1,1)}\,=0
\;\;\textrm{ and }
\cinner{q}{(1,\omega,\omega^2)}\,=0}
\]
which allows us to verify the following theorem.

\begin{theorem}
    Given $m_1, m_2, m_3>0$, the orthogonal complement of
    the space of all positive equilateral triangles
    is the complex line generated by the configuration
    \[S=(\,m_2m_3\;,\;m_1m_3\,\omega^2\;,\;m_1m_2\,\omega\,)\,.\]
\end{theorem}

\begin{corollary}
    If the masses are equal we have the orthogonal decomposition
    \[
    \C^3\simeq E^3=\Delta\oplus K\oplus K^*
    \]
    where $K^*$ denotes the space of centered and negatively oriented
    equilateral triangles.
\end{corollary}

For what follows,
we have to compute the three distances of configuration $S$ given
in the previous theorem.
Considering that the positions have
arguments that differ by $2\pi/3$, it is easy to find the distances
between the bodies.
We get that the squares of the distances in $S=(r_1,r_2,r_3)$ are
\[
\begin{array}{cc}
r_{12}^2(S)= &m_1^2m_3^2+m_2^2m_3^2+m_1m_2m_3^2\\
&\\
r_{23}^2(S)= &m_2^2m_1^2+m_3^2m_1^2+m_2m_3m_1^2\\
&\\
r_{31}^2(S)= &m_3^2m_2^2+m_1^2m_2^2+m_3m_1m_2^2
\end{array}
\]
which can also be written as
\[
r_{ij}^2=
m_i^2m_k^2+m_j^2m_k^2+m_im_jm_k^2=
m_k^2(m_i^2+m_j^2+m_im_j)\,.
\]

\begin{figure}[h]
    \centering
        \begin{tikzpicture}
        \filldraw[fill=cyan!20, line width=1pt, draw=black]
        (2.8,0)--(-1.16,2)--(-0.9,-1.56)--(2.8,0);
        \draw[red] (0,0) circle (3);
        \draw[thick] (0,0)--(4,0);
        \draw[thick] (0,0)--(-2,3.464);
        \draw[thick] (0,0)--(-2,-3.464);
        \draw[fill] (3,0) circle (0.8pt);
        \draw (3,0) node[anchor=north west] {$1$};
        \draw[fill] (-1.5,2.6) circle (0.8pt);
        \draw (-1.7,2.7) node[anchor=east] {$\omega$};
        \draw[fill] (-1.5,-2.6) circle (0.8pt);
        \draw (-1.7,-2.7) node[anchor=east] {$\omega^2$};
        \draw[fill] (2.8,0) circle (1.6pt);
        \draw (3.5,0.1) node[anchor=south] {$r_1=m_2m_3$};
        \draw[fill] (-1.16,2) circle (1.6pt);
        \draw (-1.16,2.2) node[anchor=west] {$r_3=m_1m_2\,\omega$};
        \draw[fill] (-0.9,-1.56) circle (1.6pt);
        \draw (0.2,-1.56) node[anchor=north] {$r_2=m_1m_3\,\omega^2$};
        \draw[line width=2] (0,0)--(-1.12,1.94);
        \draw[line width=2] (0,0)--(2.8,0);
        \draw[line width=2] (0,0)--(-0.9,-1.56);
        \end{tikzpicture}
    \caption{In red is the unit circle, along with the three
    cube roots of unity. In light blue is the negatively
    oriented and centered triangle $S=(r_1,r_2,r_3)$ which
    is orthogonal, for the mass inner product,
    to all positive equilateral triangles.}
    \label{fig:orthogonal.to.equilateral}
\end{figure}
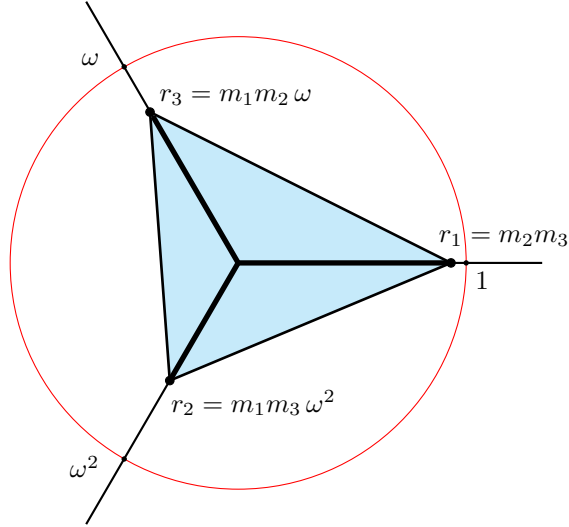

\subsection{The Hessian operator}

We recall that we are using the following notation.
The Newtonian potential for the $N$-body problem
in a Euclidean space $E=\R^d$ is
$U(x)=\sum_{i<j}m_im_j\,r_{ij}^{-1}$,
where $x=(r_1,\dots,r_N)\in E^N$ is the configuration
whose components are the positions of the bodies,
and $r_{ij}=\norm{r_i-r_j}_E$ denotes the Euclidean
distance between the bodies at positions $r_i$ and $r_j$.

We have that
\[
\frac{\partial U}{\partial r_i}(x)=
\sum_{h\neq i}m_im_h\,r_{ih}^{-3}(r_h-r_i)
=F_i(x)
\]
is the force acting on the particle at position $r_i$,
since for any $h\neq i$
\[
\frac{\partial}{\partial r_i}\,\norm{r_i-r_h}_E=
(r_i-r_h)/\norm{r_i-r_h}_E=
r_{ih}^{-1}(r_i-r_h)\,.
\]
In particular, if $\dim E=d$ and we use some orthonormal basis
of $E$ to get the identification $E\sim\R^d$, then
the force acting on the particle at position $r_i$ can be written
as $F_i=(F_i^1,\dots,F_i^d)$, where each component is
\[
F_i^k(x)=
-\sum_{h\neq i}m_im_h\,r_{ih}^{-3}(r_h^k-r_i^k)=
\frac{\partial U}{\partial r_i^k}(x)\,.
\]

To compute the second derivatives we write
\[
A_{ij}(x)=
\frac{\partial^2U}{\partial r_i\partial r_j}(x)
=\;\frac{\partial}{\partial r_j}\,\left(
\sum_{h\neq i}m_im_h\,r_{ih}^{-3}(r_h-r_i)
\right)
\]
so if $x=(r_1,\dots,r_n)$ and $u\in E$, for any $j\neq i$ we get
\begin{eqnarray*}
A_{ij}(x)(u)&=& m_im_j\,\frac{\partial}{\partial r_j}
\left(
\,r_{ij}^{-3}\,(r_j-r_i)
\right)(u)\\
&=& m_im_j
\left(\,
-3\,r_{ij}^{-5}<\,(r_j-r_i)\,,u>\,(r_j-r_i) \,+\, r_{ij}^{-3}\,u
\,\right)\,,
\end{eqnarray*}
instead, when $j=i$ we have 
\begin{eqnarray*}
A_{ii}(x)(u)&=& \sum_{h\neq i}m_im_h\,\frac{\partial}{\partial r_i}
\left(
\,r_{ih}^{-3}\,(r_h-r_i)
\right)(u)\\
&=& \sum_{h\neq i}m_im_h\,
\left(\,
3\,r_{ih}^{-5}<\,(r_h-r_i)\,,u>\,(r_h-r_i) \,-\, r_{ih}^{-3}\,u
\,\right)\\
&=& \sum_{j\neq i}\, -\,A_{ij}(x)(u)\,.
\end{eqnarray*}

Using an orthonormal basis of $E$ and performing the identification
$E\sim\R^d$, we obtain the following matrix representation
for the endomorphisms $A_{ij}(x)$.
\[
A_{ij}(x)=(\,a_{ij}^{kl}(x)\,)_{k,l=1\dots d}
\]
where we have that, for $j\neq i$,
\[
a_{ij}^{kl}(x)=m_im_j\,
(-3\,r_{ij}^{-5}\,(r_j^k-r_i^k)(r_j^l-r_i^l)\,+\, r_{ij}^{-3}\,\delta_{kl})
\]
as well as that
\[
a_{ii}^{kl}(x)=\sum_{j\neq i}\, -\,a_{ij}^{kl}(x).
\]
\subsection{The Hessian operator of an equilateral configuration}

We now begin by computing, for the case $N=3$ and $\dim E=2$,
the matrix $D^2U$ over the equilateral triangle $T=(1,\omega,\omega^2)$,
of which we know that its sides measure $\sqrt{3}$.
This matrix will be the same as for the centered triangle $T_0$,
since the Newtonian potential is translation-invariant,
which is observed in the exclusive dependence on the
relative positions of the previously found matrix.

To simplify the notation we write
\[
D^2U_{\,T}=
\left(
\begin{array}{ccc}
    -(B+C) & B & C \\
    B & -(B+D) & D \\
    C & D & -(C+D)
\end{array}\right)
\]
where
\[
B=A_{12}(T),\quad C=A_{13}(T)\quad\textrm{ and }\quad
D=A_{23}(T),
\]
and we compute the symmetric matrices $B$, $C$ and $D$.
We note first that the relative positions $r_{ij}$
in the equilateral triangle $T$ are
\[
\begin{array}{ccccccccc}
r_2-r_1&=&\omega-1&=&(&-3/2&,&\sqrt{3}/2&)\\
r_3-r_1&=&\omega^2-1&=&(&-3/2&,&-\sqrt{3}/2&)\\
r_3-r_2&=&\omega^2-\omega&=&(&0&,&-\sqrt{3}&)    
\end{array}
\]
and substituting in the formulas obtained it turns out that
{\footnotesize
\[
B=
\frac{m_1m_2}{12\sqrt{3}}
\left(
\begin{array}{cc}
    -5 & 3\sqrt{3} \\
     3\sqrt{3} & 1
\end{array}
\right)
,\;\;
C=
\frac{m_1m_3}{12\sqrt{3}}
\left(
\begin{array}{cc}
    -5 & -3\sqrt{3} \\
    -3\sqrt{3} & 1
\end{array}
\right)
\textrm{ and }
D=
\frac{m_2m_3}{12\sqrt{3}}
\left(
\begin{array}{cc}
    4 & 0 \\
    0 & -8
\end{array}
\right)
\]}

Therefore, according to the considerations made in Section
\ref{ssec:second.derivative.HU}, and in particular
the matrix expression (\ref{eq:matrix.HU})
for the Hessian endomorphism, we have that
\vspace{.5cm}

\[
\textrm{HU}_{\,T}=
\left(
\begin{array}{ccc}
    -\frac{(B+C)}{m_1} & \frac{B}{m_1} & \frac{C}{m_1} \\
    &&\\
    \frac{B}{m_2} & -\frac{(B+D)}{m_2} & \frac{D}{m_2} \\
    &&\\
    \frac{C}{m_3} & \frac{D}{m_3} & -\frac{(C+D)}{m_3}
\end{array}\right)
\]
\vspace{.5cm}

\noindent
from which we deduce that the matrix
$A=12\sqrt{3}\;\textrm{HU}_{\,T}$ is
\vspace{.5cm}
{\tiny
\[A=\left(
\begin{array}{cccccc}
5(m_2+m_3) & -3\sqrt{3}(m_2-m_3) & -5m_2 & 3\sqrt{3}m_2 & -5m_3 & -3\sqrt{3}m_3 \\
&&&&&\\
-3\sqrt{3}(m_2-m_3) & -(m_2+m_3) & 3\sqrt{3}m_2 & m_2 & -3\sqrt{3}m_3 & m_3  \\
&&&&&\\
-5m_1 & 3\sqrt{3}m_1 & 5m_1-4m_3 & -3\sqrt{3}m_1 & 4m_3 & 0  \\
&&&&&\\
3\sqrt{3}m_1 & m_1 & -3\sqrt{3}m_1 & -m_1+8m_3 & 0 & -8m_3 \\
&&&&&\\
-5m_1 & -3\sqrt{3}m_1 & 4m_2 & 0 & 5m_1-4m_2 & 3\sqrt{3}m_1 \\
&&&&&\\
-3\sqrt{3}m_1 & m_1 & 0 & -8m_2 & 3\sqrt{3}m_1 & -m_1+8m_2
\end{array}
\right)
\]}
\vspace{.3cm}

This result fully coincides with that found by Moeckel in
\cite[p. 302]{Moeckel} for the linear stability analysis
of the equilateral relative equilibrium.
In this work Moeckel computes the matrix $M^{-1}D\nabla U(T)$
which of course, is exactly our matrix $\textrm{HU}_T$.
\vspace{.3cm}

\subsection{The restriction to subspace D}
Let us now restrict the operator $HU_{\,T}$ to the two dimensional
subspace $D$, that is the complex line of $\C^3$ generated by
\[
S=(\,2m_2m_3\,,\,2m_1m_3\,\omega^2\,,\,2m_1m_2\,\omega\,)\in\C^3.
\]
The triangle $S$ corresponds to the vector
\[
\eta=
(2m_2m_3,0,-m_1m_3,-\sqrt{3}m_1m_3,-m_1m_2,\sqrt{3}m_1m_2)
\in\R^6,
\]
and the $\pi/2$ rotated triangle $iS$
corresponds to the vector
\[
\zeta=
(0,2m_2m_3,\sqrt{3}m_1m_3,-m_1m_3,-\sqrt{3}m_1m_2,-m_1m_2)
\in\R^6.
\]
Then we get
\[
A\eta=(\,m_2m_3(u_1,v_1)\,,\,m_1m_3(u_2,v_2)\,,\,m_1m_2(u_3,v_3)\,)
\]
where
\begin{eqnarray*}
z_1=u_1+i\,v_1&=&(-8m_1+10m_2+10m_3)+i\,(6\sqrt{3}(m_3-m_2))\\
z_2=u_2+i\,v_2&=&(4m_1-14m_2+4m_3)+i\,(4\sqrt{3}m_1-2\sqrt{3}m_2-8\sqrt{3}m_3)\\
z_3=u_3+i\,v_3&=&(4m_1+4m_2-14m_3)+i\,(-4\sqrt{3}m_1+8\sqrt{3}m_2+2\sqrt{3}m_3)
\end{eqnarray*}
It is immediately verified that
$\omega z_1=z_3$, $\omega z_3=z_2$ and $\omega z_2=z_1$
which confirms, as expected, that $(z_1, z_2,z_3)$ is a negative
equilateral triangle, and so that $A\eta\in D$.

In the same way we obtain
\[
A\zeta=(\,m_2m_3(u'_1,v'_1)\,,\,m_1m_3(u'_2,v'_2)\,,\,m_1m_2(u'_3,v'_3)\,)
\]
where
\begin{eqnarray*}
z'_1=u'_1+i\,v'_1&=&(6\sqrt{3}(m_3-m_2))+i\,(16m_1-2m_2-2m_3)\\
z'_2=u'_2+i\,v'_2&=&(8\sqrt{3}m_1+2\sqrt{3}m_2-4\sqrt{3}m_3)+i\,(-8m_1+10m_2-8m_3)\\
z'_3=u'_3+i\,v'_3&=&(-8\sqrt{3}m_1+4\sqrt{3}m_2-2\sqrt{3}m_3)+i\,(-8m_1-8m_2+10m_3)
\end{eqnarray*}
and again, it is verified that
$\omega z'_1=z'_3$, $\omega z'_3=z'_2$ y $\omega z'_2=z'_1$,
showing as expected the invariance of $D$ under
multiplication by $A$.

\subsection{Signature of the quadratic form}

We compute now the $2\times2$ matrix $A_D$
which is the associated matrix, for the basis $\set{\eta,\zeta}$ of $D$,
of the symmetric bilinear form given by $\inner{Au}{v}$.
It is the matrix
\[
A_D=
\left(
\begin{array}{cc}
\inner{\eta}{A\eta} & \inner{\eta}{A\zeta}  \\
\inner{\zeta}{A\eta} & \inner{\zeta}{A\zeta}
\end{array}
\right)=\left(
\begin{array}{cc}
 a & b  \\
 c & d
\end{array}
\right)
\]
where the inner products are with respect to the mass inner product,
and we obtain
\begin{align*}
a&=&2m_1m_2^2m_3^2u_1-m_1^2m_2m_3^2u_2-\sqrt{3}m_1^2m_2m_3^2v_2
-m_1^2m_2^2m_3u_3+\sqrt{3}m_1^2m_2^2m_3v_3\\
b&=&2m_1m_2^2m_3^2u'_1-m_1^2m_2m_3^2u'_2-\sqrt{3}m_1^2m_2m_3^2v'_2
-m_1^2m_2^2m_3u'_3+\sqrt{3}m_1^2m_2^2m_3v'_3\\
c&=&2m_1m_2^2m_3^2v_1+\sqrt{3}m_1^2m_2m_3^2u_2-m_1^2m_2m_3^2v_2
-\sqrt{3}m_1^2m_2^2m_3u_3-m_1^2m_2^2m_3v_3\\
d&=&2m_1m_2^2m_3^2v'_1+\sqrt{3}m_1^2m_2m_3^2u'_2-m_1^2m_2m_3^2v'_2
-\sqrt{3}m_1^2m_2^2m_3u'_3-m_1^2m_2^2m_3v'_3
\end{align*}

The trace of that matrix is
\begin{eqnarray*}
a+d&=&m_1m_2^2m_3^2(2u_1+2v'_1)+m_1^2m_2m_3^2(-u_2+\sqrt{3}u'_2-\sqrt{3}v_2-v'_2)\\
&&+m_1^2m_2^2m_3(-u_3-\sqrt{3}u'_3+\sqrt{3}v_3-v'_3)
\end{eqnarray*}
where
\begin{eqnarray*}
2u_1+2v'_1&=&16(m_1+m_2+m_3)\\
-u_2+\sqrt{3}u'_2-\sqrt{3}v_2-v'_2&=&
16(m_1+m_2+m_3)\\
-u_3-\sqrt{3}u'_3+\sqrt{3}v_3-v'_3&=&
16(m_1+m_2+m_3)
\end{eqnarray*}

Therefore we conclude that
\[
\tr A_D=\,
16\,m_1m_2m_3\,(m_1+m_2+m_3)\,(m_2m_3+m_1m_3+m_1m_2)>0.
\]

We now continue with the computation of the determinant $ad-bc$.

Factoring, we obtain:
\[
\begin{array}{cl}
a= & m_1m_2m_3\;
(\;m_2m_3(2u_1)+m_1m_3(-u_2-\sqrt{3}v_2)+m_1m_2(-u_3+\sqrt{3}v_3)\;)
\\
d= & m_1m_2m_3\;
(\;m_2m_3(2v'_1)+m_1m_3(\sqrt{3}u'_2-v'_2)+m_1m_2(-\sqrt{3}u'_3-v'_3)\;)
\end{array}
\]
where
\[
\begin{array}{ccccccc}
2u_1 &=& -u_2-\sqrt{3}v_2 &=& -u_3+\sqrt{3}v_3 &=& -16m_1+20m_2+20m_3 \\
2v'_1 &=& \sqrt{3}u'_2-v'_2 &=& -\sqrt{3}u'_3-v'_3 &=& 32m_1-4m_2-4m_3
\end{array}
\]
and therefore,
\[
\begin{array}{cl}
a= & m_1m_2m_3\,(m_2m_3+m_1m_3+m_1m_2)\,(-16m_1+20m_2+20m_3)\\
d= & m_1m_2m_3\,(m_2m_3+m_1m_3+m_1m_2)\,(32m_1-4m_2-4m_3)
\end{array}
\]
In the same way we obtain
\[
\begin{array}{cl}
b= & m_1m_2m_3\;
(\;m_2m_3(2u'_1)+m_1m_3(-u'_2-\sqrt{3}v'_2)+m_1m_2(-u'_3+\sqrt{3}v'_3)\;)
\\
c= & m_1m_2m_3\;
(\;m_2m_3(2v_1)+m_1m_3(\sqrt{3}u_2-v_2)+m_1m_2(-\sqrt{3}u_3-v_3)\;)
\end{array}
\]
where
\[
\begin{array}{ccccccc}
2u'_1 &=& -u'_2-\sqrt{3}v'_2 &=& -u'_3+\sqrt{3}v'_3 &=& 12\sqrt{3}(m_3-m_2) \\
2v_1 &=& \sqrt{3}u_2-v_2 &=& -\sqrt{3}u_3-v_3 &=& 12\sqrt{3}(m_3-m_2)
\end{array}
\]
and therefore,
\[
b=c=m_1m_2m_3\,(m_2m_3+m_1m_3+m_1m_2)\,12\sqrt{3}(m_3-m_2)
\]
Defining the positive constant $\nu$,
\[
\nu=m_1m_2m_3\,(m_2m_3+m_1m_3+m_1m_2)
\]
we can write
\[
\begin{array}{cl}
ad=& \nu^2\,(-16m_1+20m_2+20m_3)(32m_1-4m_2-4m_3) \\
&\\
=&\nu^2\,(-512m_1^2-80m_2^2-80m_3^2
+704m_1m_2+704m_1m_3-160m_2m_3)\\
&\\
bc=& \nu^2\,432 (m_3-m_2)^2\\
&\\
=&\nu^2\,(432m_3^2-864m_2m_3+432m_2^2)
\end{array}
\]
so we conclude that
\[
\nu^{-2}\det A_D=-512\,(m_1^2+m_2^2+m_3^2)+704\,(m_2m_3+m_1m_3+m_1m_2)
\]
hence that $\det A_D>0$ if and only if
\[
\lambda=
\frac{m_2m_3+m_1m_3+m_1m_2}{m_1^2+m_2^2+m_3^2}>\frac{512}{704}=\frac{8}{11}\,.
\]
Remembering that the Gascheau constant is
\[
\mu=\frac{(m_1+m_2+m_3)^2}{m_1m_2+m_2m_3+m_1m_3}
\]
we can say that $\det A_D>0$ if and only if
$\mu<27/8$, which proves Theorem \ref{thm:LagrangeisSND.iff.mu<27/8}.

\end{document}